\documentclass[10pt,reqno,oneside]{amsproc}

\usepackage{amsfonts}
\usepackage{dsfont}

\usepackage{amsmath, amsthm, amssymb, enumerate}
\usepackage{cancel}

\usepackage{color}  
\usepackage{marginnote}
\usepackage[colorlinks=true, pdfstartview=FitV, linkcolor=black, citecolor=black, urlcolor=black]{hyperref}

  \chardef\forshowkeys=0
  \chardef\refcheck=0
  \chardef\showllabel=0
  \chardef\sketches=0
  \chardef\showcolors=0
\ifnum\forshowkeys=1
  
  \usepackage[notref,notcite,color]{showkeys}
\fi

\ifnum\showllabel=1
  \def\llabel#1{\marginnote{\color{colorcccc}\rm\small(#1)}[-0.0cm]\notag}
\else
 \def\llabel#1{\notag}
\fi

\ifnum\refcheck=1
  \usepackage{refcheck}
\fi

\newtheorem{Theorem}{Theorem}[section]
\newtheorem{Corollary}[Theorem]{Corollary}
\newtheorem{Proposition}[Theorem]{Proposition}
\newtheorem{Lemma}[Theorem]{Lemma}

\theoremstyle{definition}

\newtheorem{rema}[Theorem]{Remark}

\def\andand{\text{\quad and \quad}}

\def\vk{v^{(k)}}
\def\vj{v^{(j)}}

\def\vj{v^{(j)}}

\def\inon#1{\hbox{\ \ \ \ \ \ \ }\hbox{#1}}                %in or on

\def\inin#1{\inon{in~$#1$}}
   \def\LLL#1#2{L_{\omega}^{#1}H_{x}^{#2}}

   \def\Pas{\indeq\mathbb{P}\text{-a.s.}}

   \def\PP{{\mathbb P}}

   \def\uu{{{u}}}

   \def\startnewsection#1#2{\section{#1}\label{#2}\setcounter{equation}{0}}   
   \def\NNp{{\mathbb N}}
   \def\NNz{{\mathbb N}_0}
   
    \def\TT{{\mathbb T}}
   \def\WW{{\mathbb W}}
   \def\EE{{\mathbb E}}
   \def\comma{ {\rm ,\quad{}} }            %comma in a formula
   \def\commaone{ {\rm ,\quad{}} }         %second comma in a formula
   \def\fractext#1#2{{#1}/{#2}}
   \def\indeq{\qquad{}}                     %indentation in formulas

\def\TT{\mathbb T}

\def\tilde{\widetilde}

\def\PP{\mathbb{P}}

\def\div{\mathop{\rm div}\nolimits}

\def\indeq{\quad{}}

\ifnum\showcolors=1
  \definecolor{colorcccc}{rgb}{0.7,0.7,0.7}%%%out

  \def\colb{\color{black}}%%%out
  \definecolor{colorpppp}{rgb}{0.6,0.0,0.1}
  \definecolor{colorgggg}{rgb}{.0,0.4,0.0}
  \definecolor{colorhhhh}{rgb}{0,0.6,0.2}
  \definecolor{colorgray}{rgb}{0.8,0.8,0.8}

  \definecolor{coloroftheorems}{rgb}{0.6,0.0,0.6}
  \definecolor{colorigor}{rgb}{1, 0.2, 0.8}
  \definecolor{amethyst}{rgb}{0.6, 0.4, 0.8}

  \def\cole{\color{coloroftheorems}}
  \def\colu{\color{blue}}%%%out
  \definecolor{colororange}{rgb}{0.8,0.2,0}
  \definecolor{colorpurple}{rgb}{0.6,0.0,0.6}
  
\else   %color is off
  \definecolor{colorcccc}{rgb}{0,0,0}%%%out

  \def\colb{\color{black}}%%%out
  \definecolor{colorpppp}{rgb}{0,0,0}
  \definecolor{colorgggg}{rgb}{0,0,0}
  \definecolor{colorhhhh}{rgb}{0,0,0}
  \definecolor{colorgray}{rgb}{0,0,0}

  \definecolor{coloroftheorems}{rgb}{0,0,0}
  \definecolor{colorigor}{rgb}{0,0,0}
  \definecolor{amethyst}{rgb}{0,0,0}
  
  \def\cole{\color{coloroftheorems}}
  \def\colu{\color{blue}}%%%out
  \definecolor{colororange}{rgb}{0.8,0.2,0}
  \definecolor{colorpurple}{rgb}{0.6,0.0,0.6}
  
\fi

\def\bega{\begin{aligned}}
  \def\enda{\end{aligned}}

\def\bcase{\begin{cases}}
  \def\ecase{\end{cases}}

\def\bmx{\begin{bmatrix}}
  \def\emx{\end{bmatrix}}

\def\cf{\mathcal{F}}

\def\uu{{{u}}}
\def\WW{ W}

\def\NNp{{\mathbb N}}

\def\lec{\lesssim}

\begin{document}
\baselineskip=12.3pt

$\,$
\vskip1.2truecm
\title[The stochastic Navier-Stokes equations with general $L^{3}$ data]{The stochastic Navier-Stokes equations with general $L^{3}$ data}

\author[M.S.~Ayd\i n]{Mustafa Sencer Ayd\i n}
\address{Department of Mathematics, University of Southern California, Los Angeles, CA 90089}
\email{maydin@usc.edu}

\author[I.~Kukavica]{Igor Kukavica}
\address{Department of Mathematics, University of Southern California, Los Angeles, CA 90089}
\email{kukavica@usc.edu}

\author[F.H.~Xu]{Fanhui Xu}
\address{Department of Mathematics, Union College, Schenectady, NY 12308}
\email{xuf2@union.edu}

\begin{abstract}
We consider the stochastic Navier-Stokes equations with multiplicative noise
with critical initial data. Assuming that the initial data $u_0$ belongs to
the critical space $L^{3}$ almost surely, we construct a unique local-in-time probabilistically strong solution.
We also prove an analogous result for data in the critical space~$H^\frac{1}{2}$.
\hfill 
\today
\end{abstract}

\maketitle

\date{}

\startnewsection{Introduction}{sec01}
In this paper, we address the initial value problem for the stochastic Navier-Stokes equations (SNSE)
\begin{align} 
	\begin{split}
		&\partial_t u - \Delta u  + \mathcal{P}((u\cdot\nabla) u) = \sigma(t,u) \dot{W}(t),
		\\
		&\nabla\cdot u = 0,\\
		&u(0)=u_0
		,
	\end{split}
	\label{EQ01}
\end{align}
where \(\mathcal{P}\) is the Leray projector onto divergence-free and average-free fields, and we consider the equations on the three-dimensional torus with divergence-free and average-free initial data.
In \eqref{EQ01},
the right-hand side represents a cylindrical multiplicative
It\^o noise, with $\sigma$ satisfying
\begin{align}
	%	& \Vert\sigma(t, u)\Vert_{\mathbb{H}^{\frac12 +\alpha}} \le  C(\Vert u\Vert_{H^{\frac12 +\alpha}}+1) 
	%	\comma t\ge 0
	%	\commaone \alpha= 0,\delta, \label{EQ03}  \\
	& \Vert\sigma(t, u_1)-\sigma(t, u_2)\Vert_{\mathbb{L}^{p}} 
	\le K\Vert  u_1-u_2\Vert_{L^{p}}
	\comma t\ge 0
	\commaone u_1, u_2\in L^{p}(\TT^3)
	\commaone
	p=3,6,\label{EQ04}
	\\&
	\sigma(t,0)=0 \comma t\ge 0
	\label{EQ05}
\end{align}
for some positive fixed constant~$K$.
Additionally, we assume $\div (\sigma(t, u))=0$
and $\int_{\mathbb{T}^{3}}\sigma(t,u)=0$ for $t\geq 0$ when
$\div u=0$ and $\int_{\mathbb{T}^{3}}u=0$. 

Specifically, we establish the existence and uniqueness of probabilistically strong solutions and derive an $L^3$-type energy inequality for general $L^3$ initial data. 
%Our method can also be extended to other critical spaces. 
Our method can also be extended to the Hilbert space setting.
In the final section, we develop an existence and uniqueness theory for SNSE with arbitrarily large $H^\frac{1}{2}$ initial data, demonstrating that the solutions satisfy a Sobolev-type energy inequality.

Recall that the deterministic Navier-Stokes equations (NSE) possess
an important scaling property: If $u(t,x)$ is a solution
of the NSE, so is $\lambda u(\lambda^2 t,\lambda x)$.
Based on this, in three space dimensions, the Sobolev space $W^{s,p}$ is considered
critical for initial data if $ p= 3/(1+s)$, where the Sobolev norm is invariant under scaling,
and subcritical if $p>3/(1+s)$, where the norm decays under scaling, providing more regularity and control. Similarly, the space $L^q_t L^r_x$ is critical if $(2/q) + (3/r) = 1$ and subcritical if $(2/q) + (3/r) < 1$. In the deterministic case, the existence and uniqueness in both critical and subcritical spaces have been extensively studied. 
For example,
by \cite{FJR},
for $L^{p}$ spaces with
$p>3$, there exists a unique solution on a nontrivial time interval
$[0,T]$, where $T$ depends on the size of the initial data. The same paper also contains the statement on the equivalence between a mild and a weak solution when a solution belongs to
a certain $L^{q}_tL^{r}_x$ space.
Subsequently, Kato showed in \cite{K} that
for initial data in $L^{3}$ there exists a local-in-time
solution in~$L_t^{\infty}L_x^{3}$. Furthermore, the solution is global and unique if the initial data are sufficiently small. Thus far, this appears to be the case: uniqueness has been established only under additional conditions on the initial data or the solution. For another example, according to Kato~\cite{K}, the solution is unique
if it belongs to a certain weighted-in-time Sobolev space. It is
easy to modify Kato's proof to obtain the uniqueness of solutions $u$
in the class of solutions satisfying 
$\sup_{t\in[0,T]} t^{(r-3)/(2r)}\Vert u(t,\cdot)\Vert_{L^{r}}<\infty$,
where $r>3$. Also, the uniqueness holds in the class of Leray-Hopf solutions~\cite{ESS},
if it belongs to $C_tL^{3}_x$~\cite{L},
or if the pressure $p$ is locally square integrable~\cite{Ku}. The fact that the uniqueness of solutions in $L_{t}^{\infty}L_{x}^{3}$ under general assumptions remains an open problem indicates the inherent difficulty of solving the SNSE in~$L^{3}$.

The stochastic Navier–Stokes equations exhibit similar criticality, as seen with certain types of noise, such as transport noise. In these cases, if $u$ is a local solution, then $\lambda u(\lambda^2 t,\lambda x)$, for $\lambda>0$, remains to be a local solution, with the noise $\{W_n\}_{n \geq 1}$ replaced by the rescaled noise $\{\lambda^{-1} W_n(\lambda^2 t)\}_{n \geq 1}$~\cite{AV1, AV2}. This raises the same question posed in the study of deterministic NSE: Can the critical setting be identified as the optimal one, in which the local well-posedness can be established and potentially extended to the global well-posedness?  Compared with the deterministic setting, additional difficulties in reaching the critical threshold for \eqref{EQ01} arise from the presence of multiplicative noise and use of It\^o's calculus in $L^{p}$-type spaces. When the noise is purely additive, the difference between two solutions satisfies a deterministic PDE, allowing pathwise uniqueness to be established using deterministic methods~\cite{F}. Moreover, in its mild formulation, the solution can be decomposed into the sum of an Ornstein–Uhlenbeck-type process and a nonlinear correction, where the latter satisfies a standard PDE with random coefficients. This decomposition eliminates the stochastic forcing, reducing the SPDE to a random PDE. Consequently, the existence and uniqueness theory for deterministic PDEs can be applied to establish local well-posedness. If a global-in-time energy estimate can be derived, this local well-posedness extends to the global setting~\cite{GV}. In the absence of additive noise, the challenges are further compounded when working in $L^p$-type spaces, as classical It\^o calculus was developed for Hilbert spaces. The lack of an inner product structure prevents the use of tools like the It\^o isometry, necessitating additional geometric properties of the space and more sophisticated machinery.

A substantial body of work has focused on establishing the local existence of solutions in high-regularity (i.e., subcritical) function spaces. The SNSE with additive noise was studied
by Fernando, R\"udiger, and Sritharan in \cite{FRS},
who considered initial data in $L^{p}(\mathbb{R}^d)$ with $p>d$.
The borderline case $p=d$, still in the setting of additive noise, was later addressed by Mohan and Sritharan in~\cite{MoS}.
In the case of a multiplicative noise, Krylov~\cite{Kr} analyzed general parabolic equations in $W^{2,p}(\mathbb{R}^d)$, where $p>2$, focusing on the spatial regularity. Later, Mikulevicius and Rozovski\u{\i}~\cite{MR} established the local existence of solutions to the SNSE with the initial data in $W^{1,p}(\mathbb{R}^d)$, where $p>d$. They also proved that the solution becomes global when the equation is posed in 2D. A progress on the initial regularity in the Hilbert space setting
was made by Glatt-Holtz and Ziane in~\cite{GZ},
who obtained the
local well-posedness of the SNSE with multiplicative noise on a 3D bounded domain and proved the global existence of solutions in 2D, assuming that the initial data belong to~$H^{1}$.
Subsequently, Kim \cite{Ki} constructed
local solutions to the SNSE for initial data in $H^{s}(\mathbb{R}^3)$, where $s>1/2$,
and obtained almost global solutions for small data. When referring to almost global existence, we mean that for any $p_0>0$, the solutions are global with probability $1-p_0$, provided the initial data are sufficiently small.

Regarding the initial data with critical spatial regularity,
two of the authors established in~\cite{KX2} the almost global existence
of solutions for small data in the critical space~$L^{3}$, while the
authors of the present paper proved the almost global existence for
small initial data in $H^{1/2}$ under general assumptions on the
multiplicative noise, with solutions satisfying a form of energy inequality.
Additionally, we highlight the important work of Agresti and Veraar~\cite{AV1} (see also the recent review~\cite{AV2}), who employed the maximal regularity approach to construct local solutions for initial data in a range of Besov spaces. 
For other references on strong or martingale solutions
with initial data in $L^{p}$-type spaces,
see also~\cite{AV3,AV4,BR,BT,BeT,BCF,BF,BTr,CC,CF,CFH,DZ,FS,KV,KX1,KXZ,LR,MeS,R,ZBL},
while for some well-known deterministic results, see~\cite{FK,KT,T}

The ideas in~\cite{KX2} and \cite{AKX} involve decomposing the initial data into an infinite sum of subcritical components and solving a Navier-Stokes-type system for each. The solution to the original SNSE is then constructed explicitly as the sum of these solutions. An essential ingredient in this construction is the coherence between the $L^6$ and $L^3$ norms, as well as between the $H^\frac{1}{2}$ and $H^1$ norms: The solutions to the Navier-Stokes-type system remain controlled in the subcritical space as long as their critical norm stays small. Therefore, the arguments in~\cite{KX2} and \cite{AKX} rely on the smallness of the initial data's critical norms, which is also crucial for obtaining uniform energy estimates and ensuring the summability of these solutions.

The present paper uses a data decomposition to address initial data of arbitrary size. We first decompose the initial datum into two parts: one that is small in $L^\infty_\omega L^3_x$ and the other that is regular in $L^\infty_\omega L^6_x$ (or, in the case of $H^\frac{1}{2}$ data, one part is small in $L^\infty_\omega H_x^\frac{1}{2}$ and the other is regular in $L^\infty_\omega H^1_x$). Then, we invoke~\cite[Theorem 3.1]{KXZ} (respectively~\cite[Proposition~4.2]{GZ}) to obtain a unique strong solution $\bar{w}$ for the SNSE with large subcritical data. The challenging part is constructing a unique solution $w = u - \bar{w}$ for the remainder equations (see~\eqref{EQ25} below) with small critical data. Taking the case of $L^3$ data, for example, our main strategy is to use the smallness conditions and construct $w$ under the assumption that $\bar{w} \in L^6_\omega C_t L^6_x$. When $\bar{w}=0$, this is the problem posed in~\cite{KX2}. When $\bar{w} \neq 0$, we decompose the initial data into an infinite sum of subcritical components and construct $w$ as the sum of  solutions to Navier-Stokes-type systems. This formal sum does not necessarily solve the remainder equation~\eqref{EQ25} unless uniform estimates in the desired function spaces are established to pass to the limit. The presence of $\bar{w}$ in both the drift and noise terms further complicates the analysis. By utilizing cutoff functions and stopping times, we are able to obtain uniform energy estimates. Our proof demonstrates that the potentially large $L^6$-norm of $\bar{w}$ does not have a detrimental impact on the $L^3$ dynamics.

Establishing the uniqueness of solutions for the SNSE with large data in $L_t^\infty L_x^3$ presents a distinct challenge. Recall that even in the deterministic setting this problem is known to be difficult. While uniqueness for the deterministic Navier–Stokes equations can be deduced for sufficiently small $t$ in the class $L_t^\infty L^3_x \cap L^3_t L^9_x$ using estimates
\begin{align}
	\sup_{s\le t}\Vert U(s)\Vert_{L^{3}}^3
	+\int_0^t \Vert U(s)\Vert_{L^{9}}^3 \,ds
	\lec
	\sup_{s\le t} \Vert U\Vert_{L^{3}}^3 \int_0^t (\Vert u(s)\Vert_{L^{9}}^3+\Vert \tilde{u}(s)\Vert_{L^{9}}^3)\,ds
	,\label{EQ301}
\end{align}
where $u$ and $\tilde{u}$ are solutions and $U$ denotes their difference, this argument does not extend to the stochastic setting. Instead, for the SNSE, we obtain
\begin{align}
	\EE \biggl[ \Vert U(t)\Vert_{L^{3}}^3
	+\int_0^t \Vert U(s)\Vert_{L^{9}}^3 \,ds \biggr]
	\lec
	\EE \biggl[ \sup_{s\le t} \Vert U\Vert_{L^{3}}^3 \biggr] 
	\sup_\Omega \int_0^t (\Vert u(s)\Vert_{L^{9}}^3+\Vert \tilde{u}(s)\Vert_{L^{9}}^3)\,ds
	\llabel{EQ302}
\end{align}
whose right-hand side need not be small as $t \to 0$, even if it remains bounded. Therefore, the uniqueness established in this work applies to a slightly more restrictive class of initial data; see Theorem~\ref{T01} below.

The paper is organized as follows. The next section contains preliminaries and the main statement on local-in-time solutions with $L_{\omega}^{\infty}L_x^{3}$ initial data.
Section~\ref{sec04} contains a careful construction of the solution to the remainder equations~\eqref{EQ25}.
Finally, in Section~\ref{sec06}, we state and sketch the proof of the
local existence theorem for general $H^\frac{1}{2}$ data.

\startnewsection{Preliminaries and the main results}{sec02}

To construct a probabilistically strong solution, we fix a stochastic basis 
$(\Omega, \mathcal{F}, (\mathcal{F}_t)_{t \geq 0}, \mathbb{P})$ that meets the standard assumptions. 
Let $\mathcal{H}$ and $\mathcal{Y}$ be separable Hilbert spaces, and consider 
$\{\mathbf{e}_k\}_{k \geq 1}$, a complete orthonormal basis of \(\mathcal{H}\). 
The cylindrical Wiener process over $\mathcal{H}$ is defined as 
\begin{equation*}
\WW(t, \omega) := \sum_{k \geq 1} W_k(t, \omega) \mathbf{e}_k,
\end{equation*}
where $\{W_k : k \in \mathbb{N}^+\}$ is a family of independent Brownian motions on $(\Omega, \mathcal{F}, (\mathcal{F}_t)_{t \geq 0}, \mathbb{P})$. 
We denote by $l^2(\mathcal{H}, \mathcal{Y})$ the space of Hilbert-Schmidt operators, equipped with the norm
\begin{align}
\Vert G\Vert_{l^2( \mathcal{H},\mathcal{Y})}^2:= \sum_{k=1}^{\dim \mathcal{H}} | G \mathbf{e}_k|_{\mathcal{Y}}^2.
\llabel{EQ07}
\end{align}
In this notation, the Burkholder-Davis-Gundy inequality takes the form
\begin{align}
\EE \biggl[ \sup_{s\in(0,t]}\biggl| \int_0^s G \,d\WW_r \biggr|_{\mathcal{Y}}\biggr]
\leq 
\EE\biggl[ \left(\int_0^t \Vert G\Vert^2_{ l^2( \mathcal{H},\mathcal{Y})}\, dr \right)^{1/2}\biggr].
\llabel{EQ10}
\end{align} 
\colb

For $\alpha\in \mathbb{R}$ and $p\in [1, \infty)$, we introduce the standard Sobolev spaces
  \begin{align}
   W^{\alpha,p}=\left\{f\in \mathcal{S}'(\TT^3)
   : 
   \bigl((1+|n|^2)^{\alpha/2}\hat{f}\bigr)^{\vee}\in L^p(\TT^3)
  \right\},
   \llabel{EQ63}
  \end{align}
and the Sobolev space for the noise
  \begin{align}
   \mathbb{W}^{\alpha,p}=\left\{f\colon\TT^3\to l^2( \mathcal{H},\mathcal{Y})
   : 
  f\mathbf{e}_j\in W^{\alpha, p}(\TT^3), \ j\in \mathbb{N}, \text{ and } \int_{\TT^3} \bigl\Vert \bigl((1+|n|^2)^{\alpha/2}\hat{f}\bigr)^{\vee}\bigr\Vert_{l^2( \mathcal{H},\mathcal{Y})}^p \,dx<\infty
  \right\},
  \label{EQ08}
  \end{align}
  with respective norms
   \begin{align*}
  	\Vert f\Vert_{W^{\alpha,p}}:=\Vert \bigl((1+|n|^2)^{\alpha/2}\hat{f}\bigr)^{\vee}\Vert_{L^{p}}
  	\enskip \text{ and } \enskip
  	\Vert f\Vert_{\mathbb{W}^{\alpha,p}}:=\left( \int_{\TT^3} \bigl\Vert \bigl((1+|n|^2)^{\alpha/2}\hat{f}\bigr)^{\vee}\bigr\Vert_{l^2( \mathcal{H},\mathcal{Y})}^p \,dx\right)^{1/p}.
  \end{align*}
As usual, we abbreviate $L^{p}=W^{0, p}$ and $H^{\alpha}=W^{\alpha, 2}$.
Similarly, we write $\mathbb{L}^{p}=\mathbb{W}^{0, p}$ and $\mathbb{H}^{\alpha}=\mathbb{W}^{\alpha, 2}$. To address the noise, we interpret
$( \mathcal{P}f)  \mathbf{e}_k=\mathcal{P} (f \mathbf{e}_k)$. Throughout, $C$ denotes a generic positive constant that may change from line to line.

Now, we recall the concept of a probabilistically strong solution. Let $\tau$ be a stopping time on
$(\Omega, \cf, (\cf_t)_{t\geq 0}, \PP)$ that is positive almost surely, and let $u$ be a progressively measurable, divergence-free process with respect to the same stochastic basis. For $\alpha\in \mathbb{R}$ and $p\in [1, \infty)$, we say that the pair $(\uu,\tau)$ defines a $ W^{\alpha, p}$-valued solution of~\eqref{EQ01} if $u\in  L^p(\Omega; C([0,\tau], W^{\alpha, p}))$ and 
  \begin{align}
   \begin{split}
     (u_j( t \wedge \tau),\phi)
     &=
     (u_{j, 0},\phi)
   + \int_{0}^{t\wedge\tau}
         (u_j (r), \Delta\phi )
     \,dr
   + \int_{0}^{t\wedge\tau}
      \bigl(\bigl(\mathcal{P}  (u_m(r) u(r))\bigr)_j ,\partial_{m} \phi\bigr)
     \,dr
   \\&\indeq
     +\int_0^{t\wedge\tau} \bigl(\sigma_j(r, \uu(r)),\phi\bigr)\,dW(r)
     \Pas
    \comma j=1,2,3
   \end{split}
   \label{EQ16}
  \end{align}
holds for all $\phi\in C^{\infty}(\TT^3)$ and $t\in [0,\infty)$. With this definition, we now state our main results.

\cole
\begin{Theorem}[A local-in-time solution with $L^{3}$ initial data]
\label{T01}
Assume that $\uu_0\in L^\infty(\Omega; L^{3}(\TT^3))$ satisfies
$\nabla\cdot u_0=0$ and $\int_{\TT^3} u_0=0$, and
suppose that the assumptions~\eqref{EQ04} and \eqref{EQ05} hold.
Then there exists a unique solution $(u, \tau)$
of 
\eqref{EQ01} on $(\Omega, \mathcal{F},(\mathcal{F}_t)_{t\geq
0},\mathbb{P})$, with the initial condition
$u_0$, such that $u = w+ \bar{w}$ and
  \begin{align}
  	\begin{split}
  		\EE\biggl[
  		\sup_{0\leq s\leq \tau}
  		\Vert\uu(s,\cdot)\Vert_{L^{3}}^3
  		+\int_0^{\tau_{\bar{w}}} 
  		\sum_{j=1}^3    \int_{\TT^3} | \nabla (|u_j(s,x)|^{3/2})|^2 \,dx ds
  		\biggr]
  		\leq 
  		C\sup_{\Omega} \Vert u_0\Vert_{L^{3}}^3
  	\end{split}
  	\label{EQ17}
  \end{align}
for a positive constant~$C$. In addition, 
 \begin{align}
 	\sup_{\Omega} \sup_{0\le t \le \tau} \Vert w(t)\Vert_{L^{3}} \le \epsilon
 	\label{EQ305}
 \end{align}
 for a sufficiently small constant $\epsilon$, and
 \begin{align}
 	\sup_{\Omega} \sup_{0\le t \le \tau} \Vert \bar{w}(t)\Vert_{L^{6}} \le \tilde C
 	\label{EQ306}
 \end{align}
 for some positive constant $\tilde C$, where
  \begin{align}
   \mathbb{P}[\tau>0]
   =1.
   \label{EQ13}
  \end{align}
\end{Theorem}
\colb

\colu

 \colb

The proof of Theorem~\ref{T01} is presented in Section~\ref{sec04}. Here, we outline the main idea: We construct the solution by decomposing the initial datum $u_0 \in L^\infty_\omega L^3_x$ into two divergence-free and average-free parts, 
\begin{align}
	u_0 = \bar{w}_0 + w_0
	,\label{EQ21}
\end{align} 
satisfying
\begin{align}
	\Vert w_0\Vert_{L^\infty_\omega L_x^{3}} \le \epsilon_0
	\andand
	\Vert \bar{w}_0\Vert_{L^\infty_\omega L_x^{6}} \le K_0
	,
	\label{size}
\end{align}
where $\epsilon_0$ is to be chosen small (see the proofs of Lemmas~\ref{L11}--\ref{L08} and the proof of uniqueness in Theorem~\ref{T01}) and
$K_0>0$ depends on~$\epsilon_0$. 
To achieve such decomposition, we rely on~\cite[Lemma~3.1]{KX2}, which we now state.
\cole
\begin{Lemma}[Decomposition of initial data]
	\label{L05}
	Let $u_0 \in L^\infty_\omega L^3_x$ be divergence-free and average-free.
	There exists a sequence of functions 
	\begin{equation}
		v_{0}^{(0)},v^{(1)}_{0},v_{0}^{(2)},\ldots \in \LLL{\infty}{\infty}
		\llabel{EQ024}
	\end{equation}
	that are divergence-free and average-free
	such that
	\begin{equation}
		\Vert   v_{0}^{(0)}\Vert_{L^\infty_\omega L^3_x}
		\leq 2\Vert   u_0\Vert_{L^\infty_\omega L^3_x}
		\llabel{EQ25}
	\end{equation}
	and
	\begin{equation}
		\Vert   v_{0}^{(k)}\Vert_{L^\infty_\omega L^3_x}
		\leq \frac{\Vert   u_0\Vert_{L^\infty_\omega L^3_x}}{4^{k}}
		\comma k=1,2,3,\ldots
		;
		\llabel{EQ026}
	\end{equation}
	additionally, we have
	\begin{equation}
		u_{0} = v_{0}^{(0)} + v_{0}^{(1)} + v_{0}^{(2)} + \cdots \quad\mbox{ in }L_x^3
		,
		\label{EQ027}
	\end{equation}
	almost surely.
\end{Lemma}
\colb%%%out

Now, we apply Lemma~\ref{L05} and choose $N\in\mathbb{N}$
sufficiently large so that
\begin{align}
	w_0 = \sum_{k\ge N} v_0^{(k)}
	\llabel{EQ300}
\end{align}
satisfies
$	\Vert w_0\Vert_{L^\infty_\omega L_x^{3}} \le
\epsilon_0<1/2$. Since $N$ is finite,
the remainder
$\bar{w}_0 = \sum_{k\leq N-1} v_{0}^{(k)}$
satisfies
$K_0=\Vert \bar{w}_0\Vert_{L^\infty_\omega L_x^{6}} <\infty$.
By the almost sure convergence in \eqref{EQ027},
we obtain $u_0=w_0+\bar w_0$ with~\eqref{size}.

\colb
For $\bar{w}_0$, we consider the Navier-Stokes system 
\begin{align}
	\begin{split}
		&(\partial_t - \Delta)\bar{w} 
		=-\mathcal{P} 
		(\bar{w} \cdot \nabla \bar{w}) +\sigma(t,\bar{w})\dot{W}(t),
		\\
		&
		\nabla \cdot \bar{w} = 0 \comma \bar{w}(0) =  \bar{w}_0,
		\text{ and }  \sup_\Omega\Vert \bar{w}_0\Vert_{L^6} \le K_0,
		\label{EQ22}
	\end{split} 
\end{align}
and apply
\cite[Theorem~3.1]{KXZ}, which we cite next.

\cole
\begin{Lemma}
	\label{T03}
(Local-in-time strong solution with initial data in $L^\infty_\omega L^6_x$)
	Let $T>0$. Suppose that $\bar{w}_0\in L^{\infty}(\Omega; L^6(\TT^3))$
	satisfies $\nabla\cdot \bar{w}_0=0$ and $\int_{\TT^3} \bar{w}_0=0$, and assume that the assumptions~\eqref{EQ04} and \eqref{EQ05} hold.
	Then there exist
a stopping time
	$\tau_{\bar{w}}\in(0,T)$~a.s.\ and a unique solution $(\bar{w}, \tau_{\bar{w}})$
	of 
	\eqref{EQ01} on $(\Omega, \mathcal{F},(\mathcal{F}_t)_{t\geq 0},\mathbb{P})$ with the initial condition
	$\bar{w}_0$ such that
	\begin{align}
		\begin{split}
			\EE\biggl[
			\sup_{0\leq s\leq \tau_{\bar{w}}}\Vert\bar{w}(s,\cdot)\Vert_{L^6}^6
			+\int_0^{\tau_{\bar{w}}} 
			\sum_{j=1}^3    \int_{\TT^3} | \nabla (|\bar{w}_j(s,x)|^{3})|^2 \,dx ds
			\biggr]
			\leq 
			C\EE[\vert \bar{w}_0\Vert_{L^6}^6+1],
		\end{split}
		\label{EQ15}
	\end{align}
where $C>0$ may depend on~$T$. Moreover, $\bar{w}$ satisfies 
\begin{align}
	\sup_\Omega \sup_{0\le t \le \tau_{\bar{w}}} \Vert \bar{w}(t)\Vert_{L^{6}} \leq \tilde C
	\label{EQ304}
\end{align}
for some positive constant~$\tilde C$.
\end{Lemma}
\colb

Given the existence of $\bar{w}$ asserted by Lemma~\ref{T03}, we can apply Lemma~\ref{L02} and Young's inequality to derive the following $L^3$ estimate for~$\bar{w}$.

\cole
\begin{Corollary}[$L^3$ estimate for $\bar{w}$]
	Under the assumptions of Lemma~\ref{T03}, $\bar{w}$ satisfies
	\begin{align}
		\begin{split}
			\EE\biggl[
			\sup_{0\leq s\leq \tau_{\bar{w}}}\Vert \bar{w}(s,\cdot)\Vert_{L^3}^3
			+\int_0^{\tau_{\bar{w}}} 
			\sum_{j=1}^3    \int_{\TT^3} | \nabla (|\bar{w}_j(s,x)|^{3/2})|^2 \,dx ds
			\biggr]
			\leq 
			C
			,
		\end{split}
		\label{EQ300}
	\end{align}
	where the constant $C>0$ depends on $T$ and~$K_0$ (see~\eqref{EQ22}).
\end{Corollary}
\colb

Lemma~\ref{T03} guarantees a unique local solution $\bar{w}$ to \eqref{EQ22} up to a positive stopping time~$\tau_{\bar{w}}$. Moreover, $\bar{w}$ has continuous trajectories, i.e., $\bar{w} \in L^6(\Omega;C([0,\tau_{\bar{w}}],L^6(\mathbb{T}^3)))$.
With a slight abuse of notation, we write 
\begin{equation}\label{EQ23}
\bar{w}=\bar{w}\mathds{1}_{[0, \tau_{\bar{w}} ]}+\bar{w}(\tau_{\bar{w}})\mathds{1}_{(\tau_{\bar{w}}, \infty)}. 
\end{equation}
Under this definition, $\bar{w} \in L^6(\Omega;C([0,T],L^6(\mathbb{T}^3)))$ for all $T>0$, and \eqref{EQ15} holds with $\tau_{\bar{w}}$ replaced by~$T$. Given this $\bar{w}$, we may consider the following 
Navier-Stokes-type system for a fixed $T$:
\begin{align}
	\begin{split}
&
		(\partial_t - \Delta)w 
		=-\mathcal{P} \left( 
		w \cdot \nabla w
		+ w \cdot \nabla \bar{w}
		+ \bar{w} \cdot \nabla w\right)
		+\left( \sigma(t,w+\bar{w})-\sigma(t,\bar{w})\right)\dot{W}(t),
		\\&
		\nabla \cdot w = 0,
		\\&
		w(0) =  w_0 	\Pas
		\commaone x\in\TT^3.
		\label{EQ25}
	\end{split} 
\end{align}
If a solution $(w, \tau_{w})$ to \eqref{EQ25} is shown to exist uniquely, then
\begin{align}
	u = \bar{w} + w
	\label{EQ26}
\end{align}
solves the original initial value problem~\eqref{EQ01} on $[0, \tau_{\bar{w}}\wedge \tau_{w}]$. In the next section, we address the local well-posedness of~\eqref{EQ25} and conclude the proof of Theorem~\ref{T01}, while in Section~\ref{sec06} we explain how to obtain the local well-posedness for the SNSE with general $H^{1/2}$ initial data.

\startnewsection{Truncated difference equation in $L^3$ and the proof of Theorem~\ref{T01}}{sec04}

%We express the equation as an infinite sum of difference equations to construct a unique strong solution to \eqref{EQ25} inductively.
We construct a unique strong solution to \eqref{EQ25} by expressing
it as  the sum of solutions of
difference equations and proceed inductively.
Note that when $\bar{w}=0$, \eqref{EQ25} reduces to the model in~\cite{KX2}, the SNSE with small $L^3$ data. In this section, we achieve the same construction for a general $\bar{w}\in L^6(\Omega;C[0,T],L^6(\mathbb{T}^3))$.

\cole
\begin{Proposition}
	\label{P01}
(Local-in-time solution with small initial data)
Let $T>0$, $w_0\in L^\infty(\Omega; L^3(\TT^3))$, $\nabla\cdot w_0=0$, and $\int_{\TT^3} w_0=0$, and assume that the assumptions~\eqref{EQ04}--\eqref{EQ05} hold.
In addition, supposed that $\bar{w} \in L^6(\Omega;C([0,T],L^6(\mathbb{T}^3)))$.
	Then there exist
	$\epsilon_0\in(0,1/2)$ independent of $\bar{w}$ and $\epsilon\in(\epsilon_0,1]$, such that if 
	\begin{equation}
		\Vert w_0\Vert_{L^{3}}\leq \epsilon_0
		, 
		\llabel{EQ87-2}
	\end{equation}
	then there exists
	a stopping time
	$\tau_w\in(0,T)$ and a unique solution $(w,\tau_w)$
	of 
	\eqref{EQ25} on $(\Omega, \mathcal{F},(\mathcal{F}_t)_{t\in[0,T]},\mathbb{P})$, with the initial condition
	$w_0$, such that
	\begin{align}
		\begin{split}
			\EE\biggl[
			\sup_{0\leq s\leq \tau}\Vert w(s,\cdot)\Vert_3^3
			+\int_0^{\tau} 
			\sum_{j=1}^3    \int_{\TT^3} | \nabla (|w_j(s,x)|^{3/2})|^2 \,dx ds
			\biggr]
			\leq 
			C\epsilon_0^{3}
			,
		\end{split}
		\label{EQ15-2}
	\end{align}
	where the constant $C$ may depend on $T$ and the size of~$\bar{w}$.
Moreover, $w$ satisfies
	\begin{align}
		\sup_\Omega \sup_{0\le t \le \tau} \Vert w(t)\Vert_{L^{3}} \le \epsilon
		.\label{EQ303}
	\end{align}
\end{Proposition}
\colb

Using Lemma~\ref{L05}, we write the initial data as a sum of smaller regular components, 
\begin{align}
	w_0 = \sum_{k\in \mathbb{N}_0} v_0^{(k)}
\inin{L^{3}_x},
	\label{EQ400}
\end{align}
which holds $\PP$-almost surely. Moreover,
\begin{equation}
	\sup_{\Omega} \Vert v_{0}^{(0)}\Vert_{L^3_x}
	\leq 2\epsilon_0<1,
	\label{EQ401}
 \end{equation}
and
\begin{equation}
	\sup_{\Omega}\Vert v_{0}^{(k)}\Vert_{L^3_x}
	\leq \frac{\epsilon_0}{4^{k}}
	\comma k\in \mathbb{N}.
	\label{EQ402}
\end{equation}
We expect the solution of~\eqref{EQ25} to assume the form
\begin{align}
	w = \sum_{k\in \mathbb{N}_0} v^{(k)}
	,\label{EQ403-1}
\end{align}
where $v^{(k)}$ solves
\begin{align}
	\begin{split} 
		&\partial_t \vk  -\Delta \vk
		\\&\indeq
		=
		-
		\mathcal{P}\bigl( \vk\cdot \nabla \vk\bigr)
		-
		\mathcal{P}\bigl( w^{(k-1)}\cdot \nabla \vk\bigr)
		-
		\mathcal{P}\bigl( \vk\cdot \nabla w^{(k-1)}\bigr)
		%\\&\indeq\indeq
		-
		\mathcal{P}\bigl( v^{(k)}\cdot \nabla \bar{w}\bigr)
		-
		\mathcal{P}\bigl( \bar{w}\cdot \nabla v^{(k)}\bigr)
		\\&\indeq\indeq
		+
		\sigma(t, v^{(k)} + w^{(k-1)} + \bar{w})\dot{\WW}(t)
		-
		\sigma(t, w^{(k-1)} + \bar{w})\dot{\WW}(t),
		\\&   \nabla\cdot v^{(k)}( t,x) = 0
		,
		\\&
		v^{(k)}( 0,x)=v_{0}^{(k)} (x)  \Pas
		\commaone x\in\TT^3
		,
	\end{split}
	\label{EQ404}
\end{align}
with 
\begin{equation}\label{EQ403}
	w^{(-1)}=0	\andand w^{(k-1)} = \sum_{l=0}^{k-1} v^{(l)} \comma k\in \mathbb{N}.
\end{equation}
This is because summing \eqref{EQ404} over $k$ (formally) yields the system~\eqref{EQ25} with initial datum~$w_0$. To construct $v^{(k)}$, we consider the truncated system
\begin{align}
	\begin{split}
		&\partial_t \vk  -\Delta \vk
		\\&\indeq
		=
		-
		\psi_k^2 \phi_k^2
		\mathcal{P}\bigl( \vk\cdot \nabla \vk\bigr)
		-
		\psi_k^2 \phi_k^2 \zeta_{k-1}
		\mathcal{P}\bigl( w^{(k-1)}\cdot \nabla \vk\bigr)
		-
		\psi_k^2 \phi_k^2 \zeta_{k-1}
		\mathcal{P}\bigl( \vk\cdot \nabla w^{(k-1)}\bigr)
		\\&\indeq\indeq
		-
		\psi_k^2 \phi_k^2 \psi_{\bar{w}}
		\mathcal{P}\bigl( v^{(k)}\cdot \nabla \bar{w}\bigr)
		-
		\psi_k^2 \phi_k^2 \psi_{\bar{w}}
		\mathcal{P}\bigl( \bar{w}\cdot \nabla v^{(k)}\bigr)
		\\&\indeq\indeq
		+
		\psi_k^2 \phi_k^2 \zeta_{k-1} \psi_{\bar{w}}
		\sigma(t, v^{(k)} + w^{(k-1)} + \bar{w})\dot{\WW}(t)
		-
		\psi_k^2 \phi_k^2 \zeta_{k-1} \psi_{\bar{w}}
		\sigma(t, w^{(k-1)} + \bar{w})\dot{\WW}(t),
		\\&   \nabla\cdot v^{(k)}( t,x) = 0
		,
		\\&
		v^{(k)}( 0,x)=v_{0}^{(k)} (x)  \Pas
		\commaone x\in\TT^3
		.
	\end{split}
	\label{EQ405}
\end{align}
In~\eqref{EQ405}, the functions $\psi_k$, $\phi_{k}$, and $\zeta_{k-1}$ are cutoff functions, given by
\begin{equation}
	\psi_k
	:=
	\theta\left( \frac{1}{M_k}\Vert v^{(k)}(t, \cdot)\Vert_{L^6}\right)
	\andand
	\phi_{k}
	:=
	\theta\left(
	\frac{2^k}{\epsilon_1}
	\Vert v^{(k)}(t, \cdot)\Vert_{L^3}
	\right)
	,
	\label{cutoff1}
\end{equation}
and 
\begin{equation}
	\zeta_{k-1}
	:=\mathds{1}_{k=0}+\mathds{1}_{k>0}\Pi_{i=0}^{k-1}\psi_i
	\andand
	\psi_{\bar{w}}
	:=
	\theta\left( \frac{1}{K_1}\Vert \bar{w}(t, \cdot)\Vert_{L^6}\right),
		\label{cutoff2}
\end{equation}
where $\theta\colon[0,\infty)\to [0,1]$ is a smooth function with $\theta\equiv 1$ on $[0,1]$ and 
$\theta\equiv 0$ on~$[2,\infty)$.
In addition, the constants $M_k,~\epsilon_1$ are positive values to be determined 
(see the proofs of Lemmas~\ref{L11}--\ref{L08} and the proof of uniqueness in Theorem~\ref{T01}), and $K_1$
is any fixed number greater than~$2K_0$,
where $K_0$ was introduced in~\eqref{EQ22}.
Now, we fix $T>0$ and present a lemma asserting the existence 
of a unique strong solution to~\eqref{EQ405} for any $k \in
\mathbb{N}_0$.

\cole
\begin{Lemma}[An $L^{6}$ solution]
	\label{L10}
	Let $k\in\mathbb{N}_0$ be fixed, let
	$\vk_0$ and $\bar{w}$ be as above,
	and assume that $v^{(0)}, v^{(1)}, \ldots, v^{(k-1)}$
	are given.
	Then~\eqref{EQ405} has a unique strong solution
	$\vk\in L^6(\Omega; C([0,T], L^6(\TT^3)))$, which
	satisfies
	\begin{align}
		\begin{split}
			&\EE\biggl[\sup_{0\leq t\leq T}\Vert\vk(t,\cdot)\Vert_6^6+\int_0^{T}\sum_{j=1}^{3} \int_{\TT^3} | \nabla (|\vk_j(t,x)|^{3})|^2 \,dx dt\biggr]
			%\\&\indeq
			\leq C
			\EE\biggl[
			\Vert\vk_0\Vert_6^6
			\biggr]+C
			,
		\end{split}
		\llabel{EQ406}
	\end{align}
	where the constant $C$ depends on $T$, $k$, $\epsilon_1$, and $K_1$ (see \eqref{cutoff1} and~\eqref{cutoff2}). 
\end{Lemma}
\colb

The proof of Lemma~\ref{L10} is almost identical to that of~\cite[Lemma~3.2]{KX2}. The new terms involving $\bar{w}$ are compatible with the fixed point argument presented there, thanks to~\eqref{EQ23} and the regularity of~$\bar{w}$. In addition, we can reach any $T>0$ since this fixed point argument does not rely on any
smallness assumption, and the terms involving $\bar{w}$ include the cutoff~$\psi_{\bar{w}}$.

Now, we establish uniform-in-$k$ bounds for the solutions of~\eqref{EQ405}.

\cole
\begin{Lemma}[An $L^{p}$-energy control]
	\label{L11}
	Let $T>0$, $\epsilon>0$, $p\in\{3,6\}$, and $k\in\mathbb{N}_0$. Suppose that $\vk_0$ is an initial datum of~\eqref{EQ405} satisfying \eqref{EQ400}--\eqref{EQ402}; additionally, assume that $\Vert w^{(k-1)}\Vert_3\leq\epsilon$ in $\Omega\times [0,T]$. If $\epsilon$ and $\epsilon_1$ (see \eqref{cutoff1}) are sufficiently small, then 
	\begin{align}
		\EE \left[\sup_{t\in[0, T ]} \Vert \vk(t)\Vert_p^p
		+\int_0^{ T }\sum_{j=1}^{3} \int_{\TT^3} | \nabla (|\vk_j(t,x)|^{p/2})|^2 \,dx dt
		\right]
		\leq C_{ T }\EE[\Vert \vk_0\Vert_p^p],
		\label{EQ407}
	\end{align}
	where $C>0$ is independent of~$k$.
\end{Lemma}
\colb

To prove Lemma~\ref{L11}, we require the following estimate for an established solution to the stochastic heat equation.

\cole
\begin{Lemma}[{\cite[Lemma~2.2]{KX2}}]
	\label{L02}
	Let $0<T<\infty$ and $2\leq p<\infty$. Suppose that $u_0\in L^p(\Omega; L^p(\mathbb{T}^3))$ and that $u \in L^p(\Omega; C([0,T], L^p(\mathbb{T}^3)))$ is a probabilistically strong solution to the stochastic heat equation
	\begin{align}
		\begin{split}
			&
			(\partial_t-\Delta)\uu
			=\nabla \cdot f + g\dot{\WW}(t),
			\\
			&
			\uu( 0,x)= \uu_0 ( x)
			\Pas.
		\end{split}
		\label{EQ18} 
	\end{align}
	In addition, assume that $|f|^2 |u|^{p-2} \in L^1(\Omega\times[0,T], \mathbb{L}^1(\mathbb{T}^3))$  
	and $g\in L^p(\Omega\times[0,T], \mathbb{L}^p(\mathbb{T}^3))$. 
	Then, $u$ satisfies
	\begin{align}
		\begin{split}
			&\EE\biggl[\sup_{0\leq t\leq T}\Vert\uu(t,\cdot)\Vert_{L^p}^p-\Vert\uu_0\Vert_{L^p}^p+\int_0^{T} \sum_{j=1}^3 
			\int_{\mathbb{T}^3} | \nabla (|\uu_j(t,x)|^{p/2})|^2 \,dx dt\biggr]
			\\&\indeq
			\leq C
			\EE\biggl[
			{%\colr
				\int_0^{T} \int_{\mathbb{T}^3}|f(t,x)|^2 |u(t,x)|^{p-2}\,dxdt
			}
			{%\colr
				+ \int_0^{T}\int_{\mathbb{T}^3} |\uu(t)|^{p-2} \Vert g(t,x)\Vert_{l^2}^2\,dxdt
			}
			\biggr],
		\end{split}
		\label{EQ20}
	\end{align}
	where $C$ is a positive constant that depends only on~$p$.
\end{Lemma}
\colb
\begin{rema}
We note that the same conclusion holds if instead of assuming $|f|^2 |u|^{p-2} \in L^1(\Omega\times[0,T], \mathbb{L}^1(\mathbb{T}^3))$ we consider either $f\in L^p(\Omega\times[0,T], L^q(\mathbb{T}^3))$ or $f\in L^{\frac{3p}{p+1}}(\Omega\times[0,T], L^p(\mathbb{T}^3))$. 	
\end{rema}

\begin{proof}[Proof of Lemma~\ref{L11}]
   We rewrite \eqref{EQ405}$_1$ as
   \begin{align}
   	   (\partial_t - \Delta)v^{(k)}
   	    = \nabla \cdot\left(f_{1,k}+f_{2,k}+f_{3,k}+g_{1,k}+g_{2,k}\right)+h_{k}\dot{W}
   	    ,\llabel{EQ408}
   \end{align}
where $f_{i,k}$, $i=1,2,3$, and  $g_{i,k}$, for $i=1,2$, are the terms on the first and second
 lines of the right-hand side of \eqref{EQ405}$_1$, respectively, and 
$h_k$ is the noise coefficient from the last two terms,
i.e.,
$h_k=		\psi_k^2 \phi_k^2 \zeta_{k-1} \psi_{\bar{w}}
		\sigma(t, v^{(k)} + w^{(k-1)} + \bar{w})
		-
		\psi_k^2 \phi_k^2 \zeta_{k-1} \psi_{\bar{w}}
		\sigma(t, w^{(k-1)} + \bar{w})
$.
Employing Lemma~\ref{L02}, we obtain
   \begin{align}
   	\begin{split}
   		&\EE\biggl[\sup_{0\leq t\leq T}\Vert v^{(k)}\Vert_{L^p}^p-\Vert v^{(k)}_0\Vert_{L^p}^p+\int_0^{T} \sum_{j=1}^3 
   		\int_{\mathbb{T}^3} | \nabla (|v_j^{(k)}|^{p/2})|^2 \,dx dt\biggr]
   		\\&\indeq
   		\leq C
   		\EE\biggl[
   			\int_0^{T} \int_{\mathbb{T}^3}|f_{1,k}+f_{2,k}+f_{3,k}|^2 |v^{(k)}(t,x)|^{p-2}\,dxdt\biggr]
   		\\&\indeq\indeq
   		+C\EE \biggl[
   			\int_0^{T}\int_{\mathbb{T}^3} |g_{1,k}+g_{2,k}|^2|v^{(k)}(t)|^{p-2} \,dxdt
   		\biggr]
   		+C\EE \biggl[
   		\int_0^{T}\int_{\mathbb{T}^3} |v^{(k)}(t)|^{p-2} \Vert h\Vert_{l^2}^2\,dxdt
   		\biggr]
   		\\&\indeq
   		=I_1+I_2+I_3.
   		\llabel{EQ409}
   	\end{split}
   	\end{align}
   Following the proof of~\cite[Lemma~3.3]{KX2}, we derive
   \begin{align}
   	  I_1 \le
   	   C(\epsilon+\epsilon_1)
   	    \EE \biggl[ \int_0^T \Vert v^{(k)}\Vert_{L^{3p}}^p \,dt\biggr]
   	    .\llabel{EQ410} 
   \end{align}
   For $I_2$, we utilize Hölder's and Young's inequalities, obtaining
\begin{align}
   	\begin{split}
   	   I_2
   	    &\le 
   	     \EE \biggl[ \int_0^T \int_{\TT^3} 
   	      \psi_k^4 \phi_k^4 \psi^2_{\bar{w}} |\bar{w}|^2 |v^{(k)}|^p \,dxdt\biggr]
   	      \le 
   	       \EE \biggl[ \int_0^T 
   	        \psi_k^4 \phi_k^4 \psi^2_{\bar{w}} 
   	         \Vert \bar{w}\Vert_{L^{6}}^2 \Vert v^{(k)}\Vert_{L^{\frac{3p}{2}}}^p \,dt\biggr]
   	         \\& \le 
   	          \EE \biggl[ \int_0^T 
   	          \psi_k^4 \phi_k^4 \psi^2_{\bar{w}} 
   	          \Vert \bar{w}\Vert_{L^{6}}^2 
   	           \Vert v^{(k)}\Vert_{L^{p}}^\frac{p}{2} \Vert v^{(k)}\Vert_{L^{3p}}^\frac{p}{2} \,dt\biggr]
   	            \\&\le
   	             \eta \EE \biggl[ \int_0^T 
   	              \Vert v^{(k)}\Vert_{L^{3p}}^p \,dt\biggr] 
   	             +
   	             C_\eta \EE \biggl[ \int_0^T 
   	             \psi_k^8 \phi_k^8 \psi^4_{\bar{w}} 
   	             \Vert \bar{w}\Vert_{L^{6}}^4 
   	             \Vert v^{(k)}\Vert_{L^{p}}^p \,dt\biggr]
   	        \llabel{EQ411}
   	        \end{split}
   \end{align}
  for an arbitrary $\eta>0$. To ensure the finiteness of the right-hand side, we invoke Lemma~\ref{L10} and a Poincar\'e-type inequality
   \begin{equation}\label{EQ14}
   	\Vert f\Vert_{L^{3p}}^p\leq C_p \Vert \nabla(|f|^{p/2})\Vert_{L^2}^2
   	\comma p\in[2,\infty)
   	,
   \end{equation}
   which holds for average-free functions $f$~(\cite{KZ}). Finally, we treat $I_3$ using \eqref{EQ04},
   \begin{align}
   	 	I_3 \le 
   	 	 C\EE \biggl[ \int_0^T \Vert v^{(k)}\Vert_{L^{p}}^p\,dt\biggr]
   	 .\llabel{EQ412}
   	 \end{align}
Thus, we arrive at 
   \begin{align}
   	\begin{split}
   		&\EE\biggl[\sup_{0\leq t\leq T}\Vert v^{(k)}\Vert_{L^p}^p-\Vert v^{(k)}_0\Vert_{L^p}^p+\int_0^{T} \sum_{j=1}^3 
   		\int_{\mathbb{T}^3} | \nabla (|v_j^{(k)}|^{p/2})|^2 \,dx dt\biggr]
   		\\&\indeq
   		\le C (\epsilon+\epsilon_1+\eta)
   		\EE\biggl[
   		\int_0^{T}  \Vert v^{(k)}\Vert_{L^{3p}}^p  \,dt\biggr]
   		+C_\eta (K_1^4 +1)\EE \biggl[
   		\int_0^{T} \Vert v^{(k)}\Vert_{L^{p}}^p\,dt
   		\biggr]
   		,\label{EQ413}
   	\end{split}
   \end{align}
   from where, choosing $\epsilon$, $\epsilon_1$, and $\eta$
sufficiently small, and applying the Poincaré-type inequality along with Gr\"onwall's inequality, we conclude the proof of~\eqref{EQ407}.  
\end{proof}

Next, we establish pointwise bounds on the solutions~$v^{(k)}$.

\cole
\begin{Lemma}[A pointwise $L^{3}$ control]
	\label{L09}
	Let $k\in\mathbb{N}_0$, $T>0$, and $\epsilon_1\in (2\epsilon_0, 1)$. Suppose that $\{\vk_0\}_{k\in\NNp}$ is a sequence of initial data of~\eqref{EQ405} satisfying~\eqref{EQ400}--\eqref{EQ402}. Then the solution to \eqref{EQ405} satisfies
	\begin{equation}
		\sup_{t\in [0,T]}\Vert \vk(t)\Vert_{L^{3}}
		\leq
		\frac{\epsilon_1}{2^{k-1}}
		\label{EQ71}
	\end{equation}
$\PP$-almost surely.
\end{Lemma}
\colb

Lemma~\ref{L09} follows from the proof of~\cite[Lemma 3.5]{KX2}. The main idea is that if the norm of the solution
exceeds the bound \eqref{EQ71}, it remains above the bound with positive probability
due to the continuity of its trajectories. However, as soon as the energy grows beyond
$\epsilon_1/2^{k-1}$, the cutoff functions $\phi_k$ annihilate the right-hand side
of~\eqref{EQ405}$_1$. Thus, we obtain a random heat equation for which solution's energy
dissipates almost surely, leading to a contradiction. Given \eqref{EQ403}, we note that Lemma~\ref{L09} asserts $\Vert w^{(k)}\Vert_3 \leq 4\epsilon_1$ for all $k\in\NNz$, thereby fulfilling one of the assumptions of Lemma~\ref{L11}.

Now, we define the stopping times $\tau_k$ as the first time $\Vert \vk \Vert_{L^{3}}$
reaches $\fractext{\epsilon_1}{2^{k}}$, and $\rho_k$ as the first time
$\Vert \vk \Vert_{L^{6}}$ reaches~$M_k$.
Also, we write
\begin{equation}
	\tau^{k}
	= \tau_{0} \wedge \rho_{0} \wedge \cdots \wedge \tau_{k} \wedge \rho_{k}
	\comma k\in\mathbb{N}_0
	\label{EQ85}
\end{equation}
and denote by 
\begin{equation}
	\tau_w=\inf_k \tau^k=\lim_k \tau^k
	\label{EQ70}
\end{equation}
the limiting stopping time.
It follows that 
\begin{equation}
	\sup_{t\in [0,\tau_w)}\Vert \vk(t)\Vert_{L^{3}}
	\leq
	\frac{\epsilon_1}{2^{k}}
	\comma k\in\NNz,
	\label{EQ80}
\end{equation}
and
\begin{equation}
	\sup_{t\in [0,\tau_w)}\Vert \vk(t)\Vert_{L^{6}}
	\leq
	M_k
	\comma k\in\mathbb{N}_0
	,
	\llabel{EQ81}
\end{equation}
for $k\in\mathbb{N}_0$. 
Below, we establish the almost sure positivity of~$\tau_w$.

\cole
\begin{Lemma}
	\label{L08}
Let $k\in\mathbb{N}_0$, and
assume that $\vk_0$ is an initial datum of~\eqref{EQ405} satisfying \eqref{EQ400}--\eqref{EQ402}.
	With $\tau_w$ defined as in \eqref{EQ70}, we have
	\begin{equation}
		\PP(\tau_w>0)
		= 1
		,
		\label{EQ79}
	\end{equation}
	provided that $0<\epsilon_1\ll 1$, $\epsilon_0$ is sufficiently small relative to $\epsilon_1$, and $M_k$ is sufficiently large relative to $\Vert v_0^{(k)}\Vert_{L^{\infty}_\omega L_x^6}$ (see~\eqref{cutoff1} and~\eqref{size}). 
\end{Lemma}
\colb

\begin{proof}[Proof of Lemma~\ref{L08}]
  Recall~\eqref{EQ71} and~\eqref{EQ403}. If $\epsilon_1$ is sufficiently small, then by~\eqref{EQ413} in the proof of Lemma~\ref{L11}, we can conclude that
   \begin{align}
   	\begin{split}
   		&\EE\biggl[\sup_{0\leq t\leq \delta}\Vert v^{(k)}\Vert_{L^p}^p-\Vert v^{(k)}_0\Vert_{L^p}^p 
   		\biggr] \le 
   		C\EE \biggl[
   		\int_0^{\delta} \Vert v^{(k)}\Vert_{L^{p}}^p\,dt
   		\biggr]
   		\comma p=3, 6, \label{EQ414}
   	\end{split}
   \end{align}
where $\delta \in (0, T)$ for some fixed $T>0$ and $C$ is independent of~$k$. By the definition of $\tau^k$,
   \begin{align}
   	\begin{split}
   		\PP(\tau^k<\delta)
   		&\leq
   		\sum_{j=0}^{k}
   		\mathbb{P}
   		\left(
   		\sup_{0\leq t< \delta}\Vert\vj\Vert_{L^6}^6\geq M^6_j
   		\right)
   		+
   		\sum_{j=0}^{k}
   		\mathbb{P}
   		\left(
   		\sup_{0\leq t< \delta}\Vert\vj\Vert_{L^3}^3\geq \frac{\epsilon^3_1}{2^{3j}}
   		\right)
   		\\&\leq
   		\sum_{j=0}^{k}
   		\mathbb{P}
   		\left(
   		\sup_{0\leq t< \delta}\Vert\vj\Vert_{L^6}^6 - \Vert v_0^{(j)}\Vert_{L^6}^6\geq M^6_j - \Vert v_0^{(j)}\Vert_{L^6}^6 
   		\right)
   		\\&\indeq
   		+
   		\sum_{j=0}^{k}
   		\mathbb{P}
   		\left(
   		\sup_{0\leq t< \delta}\Vert\vj\Vert_{L^3}^3 -\Vert v_0^{(j)}\Vert_{L^3}^3\geq \frac{\epsilon^3_1}{2^{3j}}-\frac{2^3\epsilon^3_0}{4^{3j}}
   		\right)
   		.
   	\end{split}
   	\llabel{EQ415}
   \end{align}
If, in addition, $\epsilon_0$ is sufficiently small compared to $\epsilon_1$, and $M_j$ is larger than $\Vert v_0^{(j)}\Vert_{L^\infty_\omega L_x^6}$ by an exponential factor of $ 2^j$, 
we may employ Markov's inequality and \eqref{EQ414}, along with \eqref{EQ400}--\eqref{EQ402}, and write
   \begin{align}
   	\begin{split}
   \PP&(\tau^k<\delta)  
   		\\&\leq
   		\sum_{j=0}^{k}
   		\frac{1}{M^6_j - \Vert v_0^{(j)}\Vert_{L^6}^6}\EE \biggl[
   		\sup_{0\leq t< \delta}\Vert\vj\Vert_{L^6}^6 - \Vert v_0^{(j)}\Vert_{L^6}^6 \biggr]
   		%\\&\indeq
   		+
   		\sum_{j=0}^{k}
   		\frac{4^{3j}}{2^{3j}\epsilon^3_1-2^3\epsilon^3_0}
   		\EE \biggl[
   		\sup_{0\leq t< \delta}\Vert\vj\Vert_{L^3}^3 -\Vert v_0^{(j)}\Vert_{L^3}^3
   		\biggr]
   		\\&\leq
   		C\delta
   		\sum_{j=0}^{k}
   		\frac{\Vert v_0^{(j)}\Vert_{L^{\infty}_\omega L_x^6}^6}{M^6_j - \Vert v_0^{(j)}\Vert_{L^{\infty}_\omega L_x^6}^6} 
   		+
   		C \delta
   		\sum_{j=0}^{k}
   		\frac{4^{3j}}{2^{3j}\epsilon^3_1-2^3\epsilon^3_0}
   		\frac{2\epsilon^3_0}{4^{3j}}
   		\le C\delta
   	.
   	\end{split}
   	\llabel{EQ416}
   \end{align}
Since $\delta>0$ is arbitrary, we obtain~\eqref{EQ79}.
\end{proof}

To prove Proposition~\ref{P01}, it remains to show that $w$, as defined in \eqref{EQ403-1}, indeed solves \eqref{EQ25} up to~$\tau_w$. The arguments are nearly identical to those in~\cite[Theorem 2.1]{KX2} and \cite[Remark 3.9]{KX2}. The terms in~\eqref{EQ404} involving $\bar{w}$ are linear in $w$ and thus introduce no additional difficulty when passing to the limit and establishing uniqueness within the class of small $L^\infty_\omega L^3_x$ solutions (see~\eqref{EQ303} which follows from \eqref{EQ80} and the additive construction of the solution). Moreover, the smallness of $\epsilon_0$ depends only on $\epsilon_1$, which is determined by the coefficients in the energy inequalities and Poincaré-type inequalities. Hence, $\epsilon_0$ is independent of~$\bar{w}$.

Applying Lemma~\ref{T03} and Proposition~\ref{P01}, we now conclude the proof of Theorem~\ref{T01}.

\begin{proof}[Proof of Theorem~\ref{T01}]
	We choose $\epsilon_0 > 0$ sufficiently small and decompose $u_0$ as $u_0 = w_0 + \bar{w}_0$ so that it satisfies~\eqref{size}. We then apply Lemma~\ref{T03} and Proposition~\ref{P01}, obtaining unique pairs $(\bar{w}, \tau_{\bar{w}})$ and $(w, \tau_w)$ solving \eqref{EQ22} and \eqref{EQ25} respectively. The existence part of Theorem~\ref{T01} follows by setting $u = w + \bar{w}$ and $\tau = \tau_w \wedge \tau_{\bar{w}}$. Since both $\tau_w$ and $\tau_{\bar{w}}$ are positive almost surely, the same holds for~$\tau$. 
	Furthermore, the energy estimate~\eqref{EQ17} is a consequence of ~\eqref{EQ300} and~\eqref{EQ15-2}.
	The bounds \eqref{EQ305} and \eqref{EQ306} follow from \eqref{EQ303} and \eqref{EQ304} upon recalling that $\tau \le \min(\tau_w, \tau_{\bar{w}})$.
	
   Now, we establish the uniqueness up to a possibly shorter stopping time,
   i.e., we show that there exists $T>0$ such that $u$ is pathwise unique on $[0,\tau \wedge T)$. 
   Let
   $(u,\tau)$ and $(\tilde{u},\tilde{\tau})$ be two solutions
   satisfying the assertions in Theorem~\ref{T01} with the same initial datum. We write $u = w + \bar{w}$ and $\tilde{u} = \tilde{w} + \tilde{\bar{w}}$. Then, $(U,\kappa)=(u-\bar{u},\tau \wedge \tilde{\tau})$
solves
\begin{align} 
   	\begin{split}
   		&(\partial_t- \Delta ) U 
   		= - \mathcal{P}((U\cdot\nabla) u)
   		- \mathcal{P}((\tilde{u}\cdot\nabla) U)
   		+ (\sigma(t,u) -\sigma(t,\tilde{u} ))\dot{W}(t),
   		\\
   		&\nabla\cdot U = 0
   		\\
   		& U(0)=0.
   	\end{split}
   	\llabel{EQ0102}
   \end{align}
Applying Lemma~\ref{L02}, we obtain
   \begin{align}
   	\begin{split}
   		&\EE\biggl[\sup_{0\leq t\leq \kappa \wedge T}\Vert U(t,\cdot)\Vert_{L^3}^3
   		+\int_0^{\kappa \wedge T} 
   		\int_{\mathbb{T}^3} | \nabla (|U|^{\frac{3}{2}})|^2 \,dx dt\biggr]
   		\\&\indeq
   		\leq C
   		\EE\biggl[
   		{%\colr
   			\int_0^{\kappa \wedge T} \int_{\mathbb{T}^3}(|U \otimes u|^2+ |\tilde{u} \otimes U|^2) |U|\,dxdt
   		}
   		{%\colr
   			+ \int_0^{\kappa \wedge T}\int_{\mathbb{T}^3} |U| \Vert \sigma(t,u)-\sigma(t,\tilde{u})\Vert_{l^2}^2\,dxdt
   		}
   		\biggr]=I_1+I_2.
   	\end{split}
   	\llabel{EQ200}
   \end{align}
For $I_2$, we have the bound
\begin{align}
	I_2
	\le
	C\EE \biggl[
	\int_0^{\kappa \wedge T} \Vert U\Vert_{L^{3}} \Vert \sigma(t,u)-\sigma(t,\tilde{u})\Vert_{\mathbb{L}^{3}}^2\,dt\biggr]
	\le
	C\EE \biggl[ \int_0^{\kappa \wedge T} \Vert U\Vert_{L^{3}}^3\,dt\biggr]
	\le
	CT\EE\biggl[ \sup_{0 \le t \le \kappa \wedge T} \Vert U\Vert_{L^3}^3\biggr] 
	.\llabel{EQ202}
\end{align}
Next, we estimate $I_1$ as
   \begin{align}
   	I_1 \le 
   	C\EE\biggl[ \int_0^{\kappa \wedge T} \int_{\mathbb{T}^3}
   	|U|^3 (|w|^2+|\tilde{w}|^2)\,dxdt \biggr]
   	+C\EE\biggl[ \int_0^{\kappa \wedge T} \int_{\mathbb{T}^3}
   	|U|^3 (|\bar{w}|^2+|\tilde{\bar{w}}|^2)\,dxdt \biggr]
   	=I_{11}+I_{12}.\llabel{EQ307}
   \end{align}
   For the first term, we employ H\"{o}lder's inequality, obtaining
   \begin{align}
   	I_{11}
   	 \le 
      C\EE \biggl[ \int_0^{\kappa \wedge T}
       \Vert U\Vert_{L^{9}}^3 (\Vert w\Vert_{L^{3}}^2+\Vert \tilde{w}\Vert_{L^{3}}^2) \,dt \biggr]
       \le 2C\epsilon^2\EE \biggl[ \int_0^{\kappa \wedge T}
       \Vert U\Vert_{L^{9}}^3 \,dt \biggr],\llabel{EQ308}
   \end{align}
   where we used \eqref{EQ305} for $w$ and $\tilde{w}$ in the last inequality.
   For $I_{12}$, utilizing H\"{o}lder's inequality again, we arrive at
   \begin{align}
   	I_{12} \le 
   	C \EE \biggl[ \int_0^{\kappa \wedge T} \Vert U\Vert_{L^{\frac{9}{2}}}^3
   	  (\Vert \bar{w}\Vert_{L^{6}}^2+\Vert \tilde{\bar{w}}\Vert_{L^{6}}^2) \,dt\biggr]
   	   \le 
   	   C\EE \biggl[ \int_0^{\kappa \wedge T} 
   	    \Vert U\Vert_{L^{9}}^\frac{3}{2}\Vert U\Vert_{L^{3}}^\frac{3}{2}
   	     (\Vert \bar{w}\Vert_{L^{6}}^2+\Vert \tilde{\bar{w}}\Vert_{L^{6}}^2)  \,dt\biggr]
   	   .\llabel{EQ309}
   \end{align}
   Now, Young's inequality and \eqref{EQ14} imply
   \begin{align}
   	\begin{split}
   	I_{12}
   	 &\le 
   	  \eta \EE \biggl[ \int_0^{\kappa \wedge T} \Vert \nabla(|U|^\frac{3}{2}|)\Vert_{L^{2}}^2\,dt \biggr] 
   	  +C_\eta
   	   \EE \biggl[ \int_0^{\kappa \wedge T} 
   	    \Vert U\Vert_{L^{3}}^3
   	   (\Vert \bar{w}\Vert_{L^{6}}^4+\Vert \tilde{\bar{w}}\Vert_{L^{6}}^4)  \,dt\biggr]
   	   \\& \le 
   	   \eta \EE \biggl[ \int_0^{\kappa \wedge T} \Vert \nabla(|U|^\frac{3}{2}|)\Vert_{L^{2}}^2\,dt \biggr] 
   	   +C_\eta
   	   T\EE \biggl[ \sup_{0\le t\le \kappa \wedge T} \Vert U\Vert_{L^{3}}^3\biggr] 
   	   ,\llabel{EQ310}
   	   \end{split}
   \end{align}
where we used \eqref{EQ306} for $\bar{w}$ and $\tilde{\bar{w}}$, and $\eta, C_\eta >0$. 
To conclude, we choose $\epsilon$, $\eta$, and $T$ sufficiently small,
from where it follows that $\PP (u(t) = \tilde{u}(t), \forall t \in [0,\tau \wedge \tilde{\tau} \wedge T])=1$, i.e., the solution is pathwise unique for a short time. 
 \end{proof}
 
Note that this proof of the uniqueness part does not rely on the uniqueness of the decomposition.
 
 \colb
\startnewsection{General $H^\frac{1}{2}$ data}{sec06}

In this section, we present an analog of Theorem~\ref{T01}
for initial data in~$H^\frac12$. First, we replace the assumption \eqref{EQ04}
by\begin{align}
	 \Vert\sigma(t, u_1)-\sigma(t, u_2)\Vert_{\mathbb{H}^{\frac12 + \delta}} 
	\le K\Vert  u_1-u_2\Vert_{\mathbb{H}^{\frac12 + \delta}}
	\commaone u_1, u_2\in \mathbb{H}^{\frac12 + \delta}(\TT^3)
	\comma t\ge 0
	\commaone
	\delta=0,\frac12,\label{EQ100}
\end{align}
and then our main result reads as follows.

\cole
\begin{Theorem}[A local solution with $H^{\frac12}$ initial data]
\label{T04}
Assume that $\uu_0\in L^\infty(\Omega; H^{\frac12}(\TT^3))$  is such that
$\nabla\cdot u_0=0$ and $\int_{\TT^3} u_0=0$,
and suppose that the assumptions~\eqref{EQ05}
and~\eqref{EQ100} hold.
	Then there exists
	%a stopping time
	%$\tau>0$ and
	a unique solution $(u, \tau)$
	of 
	\eqref{EQ01} on $(\Omega, \mathcal{F},(\mathcal{F}_t)_{t\geq
		0},\mathbb{P})$, with the initial condition
$u_0$, such that $u=w+\bar{w}$, and
	\begin{align}
		\begin{split}
			\EE\biggl[
			\sup_{0\leq s\leq \tau}
			\Vert\uu(s,\cdot)\Vert_{H^{\frac12}}^2
			+\int_0^{\tau} \Vert u\Vert_{H^{\frac{3}{2}}}^2 \,ds
			\biggr]
			\leq 
			C\sup_{\Omega} \Vert u_0\Vert_{H^{\frac12}}^2
		\end{split}
		\llabel{EQ101}
	\end{align}
for some positive constant~$C$. In addition, 
	\begin{align}
		\sup_\Omega \left(\sup_{0\le t \le \tau} \Vert w\Vert_{H^{\frac12}}^2
		 +\int_0^\tau \Vert w\Vert_{H^{\frac{3}{2}}}^2 \,dt\right)
		 \le \epsilon
		 \llabel{EQ178}
	\end{align}
for a sufficiently small constant $\epsilon$, and
	\begin{align}
		\sup_\Omega \left(\sup_{0\le t \le \tau} \Vert \bar w\Vert_{H^{1}}^2
		+\int_0^\tau \Vert \bar w\Vert_{H^{2}}^2 \,dt\right)
		\leq \tilde C
		\label{EQ178}
	\end{align}
for some positive constant $\tilde C$, where $\mathbb{P}[\tau>0]=1$.
\end{Theorem}
\colb

We utilize a similar decomposition as in Lemma~\ref{L05}, obtaining
\begin{align}
	\Vert w_0\Vert_{H^{\frac12}} \le \epsilon_0
	\andand
	\Vert \bar{w}_0\Vert_{H^{1}} \le K_0
	\Pas
	.\label{EQ103}
\end{align}
For $\bar{w}_0$, we consider the Navier-Stokes system 
\begin{align}
	\begin{split}
		&(\partial_t - \Delta)\bar{w} 
		=-\mathcal{P} 
		(\bar{w} \cdot \nabla \bar{w}) +\sigma(t,\bar{w})\dot{W}(t),
		\\
		&
		\nabla \cdot w = 0 \comma \bar{w}(0) =  \bar{w}_0,
		\text{ and }  \sup_\Omega\Vert \bar{w}_0\Vert_{H^1} \le K_0.
		\label{EQ104}
	\end{split} 
\end{align}
By~\cite[Proposition~4.2]{GZ}, it admits a unique strong solution
$(\bar{w},\tau_{\bar{w}})$ such that 
$\bar{w}\mathds{1}_{[0, \tau_{\bar{w}} ]}+\bar{w}(\tau_{\bar{w}})\mathds{1}_{(\tau_{\bar{w}}, \infty)} \in L^2(\Omega;C([0,T],H^{1}(\mathbb{T}^3))) \cap L^2(\Omega;L^2([0,T];H^{2}(\mathbb{T}^3)))$. We note that~\cite{GZ}
addresses the SNSE in bounded domains, but the proof applies to $\mathbb{T}^3$ without much change.
In addition, the solution $\bar{w}$ constructed in~\cite{GZ} need not attain the bound~\eqref{EQ178}. 
However, the energy estimate for this solution can be derived by suitably adjusting the arguments in~\cite{KXZ} and~\cite{AKX}.
\colb
Next, we consider a 
Navier-Stokes-type system \eqref{EQ25}
and set
\begin{align}
	u = \bar{w} + w
	.\llabel{EQ109}
\end{align}
To solve \eqref{EQ25}, 
we consider \eqref{EQ400}--\eqref{EQ405}
upon replacing $L^3$ by~$H^{1/2}$.
We also adjust the cutoffs by writing
\begin{equation}
	\begin{split}
	\psi_k
	&:=
	\theta\left( 
	 \frac{1}{M_k}\Vert v^{(k)}(t, \cdot)\Vert_{H^1}
	  +\frac{1}{M_k}\left(\int_0^t \Vert v^{(k)}(s, \cdot)\Vert_{H^2}^2\,ds\right)^\frac12\right)
	\\
	\phi_{k}
	&:=
	\theta\left(
	\frac{2^k}{\epsilon_1}
	\Vert v^{(k)}(t, \cdot)\Vert_{H^\frac12}
	 +\frac{2^k}{\epsilon_1}
	  \left(\int_0^t \Vert v^{(k)}(s, \cdot)\Vert_{H^\frac{3}{2}}^2\,ds\right)^\frac12
	\right)
	,
	\llabel{0cutoff1}
	\end{split}
\end{equation}
and
\begin{equation}
	\zeta_{k-1}
	:=\mathds{1}_{k=0}+\mathds{1}_{k>0}\Pi_{i=0}^{k-1}\psi_i
	\andand
	\psi_{\bar{w}}
	:=
	\theta\left( \frac{1}{K_1}\Vert v^{(k)}(t, \cdot)\Vert_{H^1}
	+\frac{1}{K_1}\left(\int_0^t \Vert v^{(k)}(s, \cdot)\Vert_{H^2}^2\,ds\right)^\frac12\right)
	\llabel{0cutoff2}
\end{equation}
where $\theta$ was given in Section~\ref{sec04}, $K_1>0$ is a sufficiently large constant,
and $M_k,~\epsilon_1>0$ are to be determined. Then,
upon utilizing~\cite[Lemma~3.1]{AKX},
we can establish the analogs of Lemmas~\ref{L10} and~\ref{L11}.
The rest of the argument follows by defining the stopping times
$\tau_{k}$ and $\rho_{k}$ 
as the first times when 
\begin{align}
	\Vert v^{(k)}(t)\Vert_{H^{\frac12}}+\left(\int_0^t\Vert v^{(k)}(s)\Vert_{H^{\frac{3}{2}}}^2\,ds\right)^\frac12
	\llabel{EQ417} 
\end{align}
and
\begin{align}
	\Vert v^{(k)}(t)\Vert_{H^{1}}+\left(\int_0^t\Vert v^{(k)}(s)\Vert_{H^{{2}}}^2\,ds\right)^\frac12
	\llabel{EQ418} 
\end{align}
hit $M_k$ and $\epsilon_1/2^k$, respectively. Finally, the uniqueness follows similarly to the proof of Theorem~\ref{T01}, since $H^{1/2}$ and $H^1$ has the same scaling properties as $L^3$ and $L^6$, respectively.

\colb
\section*{Acknowledgments}
\rm
MA and IK were supported in part by the NSF grant DMS-2205493.

\ifnum\sketches=1

\fi

\end{document}